\documentclass{aptpub}
\usepackage{bbm, graphicx}

\authornames{Richard Kenyon and Mei Yin} 
\shorttitle{On the asymptotics of constrained exponential random graphs} 

\newcommand{\PR}{\mathbb P}
\newcommand{\G}{\mathcal{G}}
\newcommand{\eps}{\epsilon}

\newcommand{\mytilde}{\raise.17ex\hbox{$\scriptstyle\mathtt{\sim}$}}

\numberwithin{equation}{section}  

\begin{document}

\title{On the asymptotics of constrained exponential random graphs} 

\authorone[Brown University]{Richard Kenyon}

\addressone{Department of Mathematics, Brown University, Providence, RI 02912, USA}
\emailone{rkenyon@math.brown.edu}

\authortwo[University of Denver]{Mei Yin} 

\addresstwo{Department of Mathematics, University of Denver, Denver, CO 80208, USA}
\emailtwo{mei.yin@du.edu} 

\begin{abstract}
The unconstrained exponential family of random graphs assumes no
prior knowledge of the graph before sampling, but it is natural to consider
situations where partial information about the graph is known, for example the total number of edges.
What does a typical
random graph look like, if drawn from an exponential model subject to such
constraints? Will there be a similar phase
transition phenomenon (as one varies the parameters) as that which occurs in the unconstrained
exponential model? We present some general results for this
constrained model and then apply them to get concrete answers in
the edge-triangle model with fixed density of edges.
\end{abstract}



\section{Introduction}
\label{1} Consider the set $\mathcal{G}_n$ of all simple graphs
$G_n$ on $n$ vertices (``simple'' means undirected, with no loops
or multiple edges). By a $k$-parameter family of exponential
random graphs we mean a family of probability measures
$\PR_n^{\beta}$ on $\G_n$ defined by, for $G_n\in\G_n$,
\begin{equation}
\label{pmf} \PR_n^{\beta}(G_n)=\exp\left[n^2\left(\beta_1
t(H_1,G_n)+\cdots+
  \beta_k t(H_k,G_n)-\psi_n^{\beta}\right)\right],
\end{equation}
where $\beta=(\beta_1,\dots,\beta_k)$ are $k$ real parameters,
$H_1,\dots,H_k$ are pre-chosen finite simple graphs (and we take
$H_1$ to be a single edge), $t(H_i, G_n)$ is the density of graph
homomorphisms (the probability that a random vertex map $V(H_i)
\to V(G_n)$ is edge-preserving),
\begin{equation}
\label{t} t(H_i, G_n)=\frac{|\text{hom}(H_i,
G_n)|}{|V(G_n)|^{|V(H_i)|}},
\end{equation}
and $\psi_n^{\beta}$ is the normalization constant,
\begin{equation}
\label{psi} \psi_n^{\beta}=\frac{1}{n^2}\log\sum_{G_n \in
\mathcal{G}_n} \exp\left[n^2 \left(\beta_1
t(H_1,G_n)+\cdots+\beta_k t(H_k,G_n)\right) \right].
\end{equation}
Sometimes, other than homomorphism densities, we also consider
more general bounded continuous functions on the graph space (a
notion to be made precise later), for example the degree sequence
or the eigenvalues of the adjacency matrix.

Exponential random graphs have been used to model real-world
networks as they are able to capture a wide variety of common
\emph{network tendencies} by representing a complex global
structure through a set of tractable local features \cite{FS}
\cite{HJ} \cite{N} \cite{Hofstad} \cite{WF}. Intuitively, we can
think of the $k$ parameters $\beta_1,\ldots,\beta_k$ as tuning
parameters that allow one to adjust the influence of different
subgraphs $H_1,\ldots,H_k$ of $G_n$ on the probability
distribution, whose asymptotics are our main interest since
networks are often very large in size. As flexible as they are,
exponential models admittedly have one shortcoming: they are
centered on dense graphs whereas most network data in the real
world are sparse. In this sense, one could argue that exponential
random graphs (and the graphon technology developed by Lov\'{a}sz
et al. \cite{BCLSV1} \cite{BCLSV2} \cite{BCLSV3} \cite{L}
\cite{LS} that is heavily used in studying them) are of limited
relevance in studying real networks. However, from the point of
view of extremal combinatorics and statistical mechanics,
exponential random graphs and constrained graphons represent an
important and challenging class of models, displaying both diverse
and novel phase transition behavior \cite{RRS} \cite{RS1}
\cite{RS} \cite{RY}.

Our main results are (Theorem \ref{main1}) a variational principle
for the normalization constant (partition function) for graphons
with constrained edge density, and an associated concentration of
measure (Theorem \ref{main2}) indicating that almost all large
constrained graphs lie near the maximizing set. We then specialize
to the edge-triangle model, and show the existence of
first-order phase transitions in the edge-density constrained
models.

\section{Background}

We begin by reviewing some notation and results concerning the
theory of graph limits and its use in exponential random graph
models. Following the earlier work of Aldous \cite{Aldous} and
Hoover \cite{Hoover}, Lov\'{a}sz and coauthors (V.T. S\'{o}s, B.
Szegedy, C. Borgs, J. Chayes, K. Vesztergombi,...) have
constructed an elegant theory of graph limits in a sequence of
papers \cite{BCLSV1} \cite{BCLSV2} \cite{BCLSV3} \cite{LS}. See
also the recent book \cite{L} for a comprehensive account and
references. This sheds light on various topics such as graph
testing and extremal graph theory, and has found applications in
statistics and related areas (see for instance \cite{CDS}). Though
their theory has been developed for dense graphs (number of edges
comparable to the square of number of vertices), serious attempts
have been made at formulating parallel results for sparse graphs
\cite{Bo} \cite{BCCZ}.

Here are the basics of this beautiful theory. Any simple graph
$G_n$, irrespective of the number of vertices, may be represented
as an element $h^{G_n}$ of a single abstract space $\mathcal{W}$
that consists of all symmetric measurable functions from $[0,1]^2$
into $[0,1]$, by defining
\begin{equation}
h^{G_n}(x,y)=\left\{%
\begin{array}{ll}
    1, & \hbox{if $(\lceil nx \rceil, \lceil ny \rceil)$ is an edge in $G_n$;} \\
    0, & \hbox{otherwise.} \\
\end{array}%
\right.
\end{equation}
A sequence of graphs $\{G_n\}_{n\geq 1}$ is said to converge to a
function $h \in \mathcal{W}$ (referred to as a ``graph limit'' or
``graphon'') if for every finite simple graph $H$ with vertex set
$V(H)=[k]=\{1,...,k\}$ and edge set $E(H)$,
\begin{equation}
\lim_{n \to \infty} t(H, h^{G_n})=t(H, h),
\end{equation}
where $t(H, h^{G_n})=t(H, G_n)$, the graph homomorphism density
(\ref{t}), by construction, and
\begin{equation}
\label{tt} t(H, h)=\int_{[0,1]^k}\prod_{\{i,j\}\in E(H)}h(x_i,
x_j)dx_1\cdots dx_k.
\end{equation}
Indeed every function in $\mathcal{W}$ is the limit of a certain
convergent graph sequence \cite{LS}. Intuitively, the interval
$[0,1]$ represents a ``continuum'' of vertices, and $h(x,y)$
denotes the probability of putting an edge between $x$ and $y$.
For example, for the Erd\H{o}s-R\'{e}nyi random graph $G(n,
\rho)$, the ``graphon'' is represented by the function that is
identically equal to $\rho$ on $[0,1]^2$. This ``graphon''
interpretation enables us to capture the notion of convergence in
terms of subgraph densities by an explicit metric on
$\mathcal{W}$, the so-called ``cut distance'':
\begin{equation}
d_{\square}(f, h)=\sup_{S, T \subseteq [0,1]}\left|\int_{S\times
T}\left(f(x, y)-h(x, y)\right)dx\,dy\right|
\end{equation}
for $f, h \in \mathcal{W}$. A non-trivial complication is that the
topology induced by the cut metric is well defined only up to
measure preserving transformations of $[0,1]$ (and up to sets of
Lebesgue measure zero), which in the context of finite graphs may
be thought of as vertex relabeling. To tackle this issue, an
equivalence relation $\sim$ is introduced in $\mathcal{W}$. We say
that $f\sim h$ if $f(x, y)=h_{\sigma}(x, y):=h(\sigma x, \sigma
y)$ for some measure preserving bijection $\sigma$ of $[0,1]$. Let
$\tilde{h}$ (referred to as a ``reduced graphon'') denote the
equivalence class of $h$ in $(\mathcal{W}, d_{\square})$. Since
$d_{\square}$ is invariant under $\sigma$, one can then define on
the resulting quotient space $\tilde{\mathcal{W}}$ the natural
distance $\delta_{\square}$ by $\delta_{\square}(\tilde{f},
\tilde{h})=\inf_{\sigma_1, \sigma_2}d_{\square}(f_{\sigma_1},
h_{\sigma_2})$, where the infimum ranges over all measure
preserving bijections $\sigma_1$ and $\sigma_2$, making
$(\tilde{\mathcal{W}}, \delta_{\square})$ into a metric space.
With some abuse of notation we also refer to $\delta_{\square}$ as
the ``cut distance''. The space $(\tilde{\mathcal{W}}, \delta_{\square})$ enjoys many
important properties that are essential for the study of
exponential random graph models. For example, it is a compact
space and homomorphism densities $t(H, \cdot)$ are continuous
functions on it.

For the purpose of this paper, two theorems from Chatterjee and
Diaconis \cite{CD} (both based on a large deviation result
established in Chatterjee and Varadhan \cite{CV}) merit some
special attention. Together they connect the occurrence of a phase
transition in the exponential model with the solution of a certain
maximization problem. Their results are formulated in terms of
general exponential models where the terms in the exponent
defining the probability measure may contain functions on the
graph space other than homomorphism densities, as alluded to at
the beginning of this paper. Let $T: \tilde{\mathcal{W}}
\rightarrow \mathbb{R}$ be a bounded continuous function. Let the
probability measure $\PR_n$ and the normalization constant
$\psi_n$ be defined as in (\ref{pmf}) and (\ref{psi}), that is,
\begin{equation}
\label{pmf2}
\PR_n(G_n)=\exp\left(n^2(T(\tilde{h}^{G_n})-\psi_n)\right),
\end{equation}
\begin{equation}
\label{psi2} \psi_n=\frac{1}{n^2}\log\sum_{G_n \in \mathcal{G}_n}
\exp\left(n^2 T(\tilde{h}^{G_n}) \right).
\end{equation}
The first theorem (Theorem 3.1 in \cite{CD}) states that the
limiting normalization constant $\psi:=\lim_{n \rightarrow \infty}
\psi_n$ of the exponential random graph, which is crucial for the
computation of maximum likelihood estimates, always exists and is
given by
\begin{equation}
\label{setmax} \psi=\sup_{\tilde{h}\in
\tilde{W}}\left(T(\tilde{h})-I(\tilde{h})\right),
\end{equation}
where $I$ is first defined as a function from $[0, 1]$ to
$\mathbb{R}$ as
\begin{equation}
\label{I1} I(u)=\frac{1}{2}u\log u+\frac{1}{2}(1-u)\log(1-u),
\end{equation}
and then extended to $\tilde{\mathcal{W}}$ in the usual manner:
\begin{equation}
\label{I2} I(\tilde{h})=\int_{[0, 1]^2}I(h(x, y))\,dx\,dy,
\end{equation}
where $h$ is any representative element of the equivalence class
$\tilde{h}$. It was shown in \cite{CV} that $I$ is well defined
and lower semi-continuous on $\tilde{\mathcal{W}}$. Let
$\tilde{H}$ be the subset of $\tilde{\mathcal{W}}$ where $\psi$ is
maximized. By the compactness of $\tilde{\mathcal{W}}$, the
continuity of $T$ and the lower semi-continuity of $I$,
$\tilde{H}$ is a nonempty compact set. The set $\tilde{H}$ encodes
important information about the exponential model (\ref{pmf2}) and
helps to predict the behavior of a typical random graph sampled
from this model. The second theorem (Theorem 3.2 in \cite{CD})
states that in the large $n$ limit, the quotient image
$\tilde{h}^{G_n}$ of a random graph $G_n$ drawn from (\ref{pmf2})
must lie close to $\tilde{H}$ with high probability,
\begin{equation}
\delta_{\square}(\tilde{h}^{G_n}, \tilde{H})\to 0\hbox{ in
probability as } n\to \infty.
\end{equation}
Since the limiting normalization constant $\psi$ is the generating
function for the limiting expectations of other random variables
on the graph space such as expectations and correlations of
homomorphism densities, a phase transition occurs when $\psi$ is
non-analytic or when $\tilde{H}$ is not a singleton set.

\section{Constrained exponential random graphs}
\label{2} The exponential family of random graphs introduced above
have popular counterparts in statistical physics: a hierarchy of
models ranging from the grand canonical ensemble, the canonical
ensemble, to the microcanonical ensemble, with subgraph densities
in place of particle and energy densities, and tuning parameters
in place of temperature and chemical potentials. In the grand
canonical ensemble, the exponential model (\ref{pmf}) in this
case, no prior knowledge of the graph is assumed. As useful as
they are, for large networks these models are sometimes
inappropriate. For example, as shown by Chatterjee and Diaconis
\cite{CD}, when $k=2$ and $\beta_2>0$, all graphs drawn from
(\ref{pmf}) where $H_1$ is an edge and $H_2$ is any finite simple
graph are not appreciably different from Erd\H{o}s-R\'{e}nyi in
the large $n$ limit. This somewhat trivial conclusion implies that
sometimes subgraph densities cannot be tuned and exponential
random graphs alone may not capture all desirable features of the
networked system, such as interdependency and clustering. We are
thus motivated to study variants of the exponential random graph
model: the canonical ensemble, where some subgraph density is
controlled directly and others are tuned with parameters, and the
microcanonical ensemble, where complete information of the graph
is observed beforehand.

One difficulty arises. Unlike standard statistical physics models,
the equivalence of various ensembles in the asymptotic regime does
not hold in these models (see \cite{TET} for discussions about
non-equivalence of ensembles due to non-concavity of entropy). A
natural question to ask is what would be a typical random graph
drawn from an exponential model subject to certain constraints? Or
perhaps more importantly will there be a similar phase transition
phenomenon as in the standard exponential model (hereby referred
to as an ``unconstrained model'')? The following Theorems
\ref{main1} and \ref{main2} give a detailed answer to these
questions. Not surprisingly, the proofs follow a similar line of
reasoning as in Theorems 3.1 and 3.2 of \cite{CD}. However, there
are noted differences in how we interpret these phase transition
results. For example, a typical graph drawn from the constrained
edge-triangle model still exhibits Erd\H{o}s-R\'{e}nyi structure
for $\beta_2$ close to $0$, but consists of one big clique and
some isolated vertices as $\beta_2$ gets sufficiently close to
infinity, so the transition is between graphs of different
characters. In the unconstrained model, on the other hand,
although there is a curve in the parameter space across which the
graph densities display sudden jumps \cite{CD} \cite{RY}, the
transition is between graphs of similar characters
(Erd\H{o}s-R\'{e}nyi graphs). This gives one more reason why the
constrained model deserves its own attention. Due to the imposed
constraints, instead of working with probability measure $\PR_n$
and normalization constant $\psi_n$ as in \cite{CD}, we are
working with conditional probability measure and conditional
normalization constant, so the argument is more involved. The
proof of Theorem \ref{main1} also incorporates some ideas from
Theorem 3.1 of \cite{RS1}.

For clarity, we assume that the edge density of the graph is
approximately known, though the proof runs through without much modification if the density of some
other more complicated subgraph is approximately described. We make precise the
notion of ``approximately'' below. We still assign a
probability measure $\PR_n$ as in (\ref{pmf2}) on $\G_n$, but we
will consider a conditional normalization constant and also define
a conditional probability measure. Let $e\in[0,1]$ be a real
parameter that signifies an ``ideal'' edge density. Take
$\alpha>0$. The conditional normalization constant
$\psi^{e}_{n,\alpha}$ is defined analogously to the normalization
constant for the unconstrained exponential random graph model,
\begin{equation}
\label{cpsi1} \psi^{e}_{n,\alpha}=\frac{1}{n^2}\log\sum_{G_n\in
\mathcal{G}_n: |e(G_n)-e|<\alpha}\exp\left(n^2
T(\tilde{h}^{G_n})\right),
\end{equation}
the difference being that we are only taking into account graphs
$G_n$ whose edge density $e(G_n)$ is within an $\alpha$
neighborhood of $e$. Correspondingly, the associated conditional
probability measure $\PR^{e}_{n,\alpha}(G_n)$ is given by
\begin{equation}
\label{cpmf} \PR^{e}_{n,\alpha}(G_n)=\exp(-n^2
\psi^{e}_{n,\alpha})\exp\left(n^2
T(\tilde{h}^{G_n})\right)\mathbbm{1}_{|e(G_n)-e|<\alpha}.
\end{equation}
We perform two limit operations on $\psi^{e}_{n,\alpha}$. First we
take $n$ to infinity, then we shrink the interval around $e$ by
letting $\alpha$ go to zero:
\begin{equation}
\label{cpsi2} \psi^{e}=\lim_{\alpha \rightarrow 0}\lim_{n
\rightarrow \infty}\psi^{e}_{n,\alpha}.
\end{equation}
Intuitively, these two operations ensure that we are examining the
asymptotics of exponentially weighted large graphs with edge
density sufficiently close to $e$. Theorem \ref{main1} shows that
this is indeed the case.

\begin{theorem}
\label{main1} Let $e: 0\leq e\leq 1$ be a real parameter and $T:
\tilde{\mathcal{W}} \rightarrow \mathbb{R}$ be a bounded
continuous function. Let $I$ and $\psi^{e}$ be defined as before
(see (\ref{I1}), (\ref{I2}), (\ref{cpsi1}) and (\ref{cpsi2})).
Then
\begin{equation}
\label{csetmax} \psi^{e}=\sup_{\tilde{h}\in \tilde{\mathcal{W}}:
e(\tilde{h})=e}\left(T(\tilde{h})-I(\tilde{h})\right),
\end{equation}
where
\begin{equation}
e(\tilde{h})=\int_{[0,1]^2}h(x_1, x_2)dx_1dx_2,
\end{equation}
and $h$ is any function in the equivalence class $\tilde{h}$.
\end{theorem}

\begin{proof}
By definition, $\liminf \psi^{e}_{n,\alpha}$ and $\limsup
\psi^{e}_{n,\alpha}$ exist as $n \rightarrow \infty$. We will show
that they both approach $\sup_{\tilde
h:e(\tilde{h})=e}(T(\tilde{h})-I(\tilde{h}))$ as $\alpha
\rightarrow 0$. For this purpose we need to define a few sets. Let
$\tilde{U}_{\alpha}$ be the open strip of reduced graphons
$\tilde{h}$ with $e-\alpha<e(\tilde{h})<e+\alpha$, and let
$\tilde{H}_{\alpha}$ be the closed strip $e-\alpha \leq
e(\tilde{h}) \leq e+\alpha$. Fix $\epsilon>0$. Since $T$ is a
bounded function, there is a finite set $R$ such that the
intervals $\{(c, c +\epsilon): c\in R\}$ cover the range of $T$.
For each $c\in R$, let $\tilde{U}_{\alpha,c}$ be the open set of
reduced graphons $\tilde{h}$ with $e-\alpha<e(\tilde{h})<e+\alpha$
and $c<T(\tilde{h})<c+\epsilon$, and let $\tilde{H}_{\alpha,c}$ be
the closed set $e-\alpha \leq e(\tilde{h}) \leq e+\alpha$ and
$c\leq T(\tilde{h})\leq c+\epsilon$. It may be assumed without
loss of generality that $\tilde{U}_{\alpha,c}$ and
$\tilde{H}_{\alpha,c}$ are nonempty for each $c \in R$. Let
$|\tilde{U}_{\alpha,c}^n|$ and $|\tilde{H}_{\alpha,c}^n|$ denote
the number of graphs with $n$ vertices whose reduced graphons lie
in $\tilde{U}_{\alpha,c}$ or $\tilde{H}_{\alpha,c}$, respectively.
The large deviation principle, Theorem 2.3 of \cite{CV}, implies
that:
\begin{equation}
\limsup_{n\rightarrow \infty}\frac{\log
|\tilde{H}_{\alpha,c}^n|}{n^2}\leq -\inf_{\tilde{h}\in
\tilde{H}_{\alpha,c}}I(\tilde{h}),
\end{equation}
and that
\begin{equation}
\liminf_{n\rightarrow \infty}\frac{\log
|\tilde{U}_{\alpha,c}^n|}{n^2}\geq -\inf_{\tilde{h}\in
\tilde{U}_{\alpha,c}}I(\tilde{h}).
\end{equation}

We first consider $\limsup \psi^{e}_{n,\alpha}$.
\begin{equation}
\exp(n^2\psi^{e}_{n,\alpha})\leq \sum_{c\in
R}\exp(n^2(c+\epsilon))|\tilde{H}_{\alpha,c}^n|\leq |R|\sup_{c\in
R}\exp(n^2(c+\epsilon))|\tilde{H}_{\alpha,c}^n|.
\end{equation}
This shows that
\begin{equation}\label{supce}
\limsup_{n\rightarrow \infty}\psi^{e}_{n,\alpha}\leq \sup_{c\in
R}\left(c+\epsilon-\inf_{\tilde{h}\in
\tilde{H}_{\alpha,c}}I(\tilde{h})\right).
\end{equation}
Each $\tilde{h}\in \tilde{H}_{\alpha,c}$ satisfies
$T(\tilde{h})\geq c$. Consequently,
\begin{equation}
\sup_{\tilde{h}\in
\tilde{H}_{\alpha,c}}(T(\tilde{h})-I(\tilde{h}))\geq
\sup_{\tilde{h}\in
\tilde{H}_{\alpha,c}}(c-I(\tilde{h}))=c-\inf_{\tilde{h}\in
\tilde{H}_{\alpha,c}}I(\tilde{h}).
\end{equation}
Substituting this in (\ref{supce}) gives
\begin{eqnarray}
\limsup_{n\rightarrow \infty}\psi^{e}_{n,\alpha}&\leq&
\epsilon+\sup_{c\in R}\sup_{\tilde{h}\in
\tilde{H}_{\alpha,c}}(T(\tilde{h})-I(\tilde{h}))\\\notag&=&\epsilon+\sup_{\tilde{h}\in
\tilde{H}_{\alpha}}(T(\tilde{h})-I(\tilde{h})).
\end{eqnarray}

Next we consider $\liminf \psi^{e}_{n,\alpha}$.
\begin{equation}
\exp(n^2\psi^{e}_{n,\alpha})\geq \sup_{c\in R}
\exp(n^2c)|\tilde{U}_{\alpha,c}^n|.
\end{equation}
Therefore for each $c\in R$,
\begin{equation}\label{infc}
\liminf_{n \rightarrow \infty}\psi^{e}_{n,\alpha}\geq
c-\inf_{\tilde{h}\in \tilde{U}_{\alpha,c}}I(\tilde{h}).
\end{equation}
Each $\tilde{h}\in \tilde{U}_{\alpha,c}$ satisfies
$T(\tilde{h})<c+\epsilon$. Therefore
\begin{equation}
\sup_{\tilde{h}\in
\tilde{U}_{\alpha,c}}(T(\tilde{h})-I(\tilde{h}))\leq
\sup_{\tilde{h}\in
\tilde{U}_{\alpha,c}}(c+\epsilon-I(\tilde{h}))=c+\epsilon-\inf_{\tilde{h}\in
\tilde{U}_{\alpha,c}}I(\tilde{h}).
\end{equation}
Together with (\ref{infc}), this shows that
\begin{eqnarray}
\liminf_{n \rightarrow \infty}\psi^{e}_{n,\alpha}&\geq&
-\epsilon+\sup_{c\in R}\sup_{\tilde{h}\in
\tilde{U}_{\alpha,c}}(T(\tilde{h})-I(\tilde{h}))\\\notag&=&-\epsilon+\sup_{\tilde{h}\in
\tilde{U}_{\alpha}}(T(\tilde{h})-I(\tilde{h})).
\end{eqnarray}

Since $\epsilon$ is arbitrary, this yields a chain of inequalities
\begin{multline}
\label{chain} \sup_{\tilde{h}\in
\tilde{H}_{\alpha-\alpha^2}}(T(\tilde{h})-I(\tilde{h}))\leq
\sup_{\tilde{h}\in
\tilde{U}_{\alpha}}(T(\tilde{h})-I(\tilde{h}))\leq \liminf_{n
\rightarrow \infty}\psi^{e}_{n,\alpha}\\\leq \limsup_{n
\rightarrow \infty}\psi^{e}_{n,\alpha}\leq \sup_{\tilde{h}\in
\tilde{H}_{\alpha}}(T(\tilde{h})-I(\tilde{h})).
\end{multline}
As $\alpha \rightarrow 0+$, the limits of
$\sup_{\tilde{H}_{\alpha-\alpha^2}}(T(\tilde{h})-I(\tilde{h}))$
and $\sup_{\tilde{H}_{\alpha}}(T(\tilde{h})-I(\tilde{h}))$ are the
same, so we have proven that
\begin{equation}
\label{psie} \psi^{e}=\lim_{\alpha \rightarrow 0}\lim_{n
\rightarrow \infty}\psi^{e}_{n,\alpha}=\lim_{\alpha\rightarrow
0}\sup_{\tilde{h}\in
\tilde{H}_{\alpha}}(T(\tilde{h})-I(\tilde{h})).
\end{equation}
First we establish that the right-hand side of (\ref{psie}) is
equal to $\sup_{\tilde{H}_0}(T(\tilde{h})-I(\tilde{h}))$, where
$\tilde{H}_0=\{\tilde{h}: e(\tilde{h})=e\}$. By the compactness of
$\tilde{\mathcal{W}}$ and the continuity of $e$, $\tilde{H}_0$ is
a nonempty compact set. By definition, we can find a sequence of
reduced graphons $\tilde{h}_{\alpha} \in \tilde{H}_{\alpha}$ such
that $\lim_{\alpha \rightarrow
0}(T(\tilde{h}_{\alpha})-I(\tilde{h}_{\alpha}))=\lim_{\alpha\rightarrow
0}\sup_{\tilde{H}_{\alpha}}(T(\tilde{h})-I(\tilde{h}))$. These
reduced graphons converge to a reduced graphon $\tilde{h}_0 \in
\tilde{H}_0$. Since $T$ is continuous and $I$ is lower
semi-continuous,
\begin{equation}
\sup_{\tilde{H}_0} (T(\tilde{h})-I(\tilde{h})) \geq
T(\tilde{h}_0)-I(\tilde{h}_0) \geq \lim_{\alpha \rightarrow
0}(T(\tilde{h}_{\alpha})-I(\tilde{h}_{\alpha})).
\end{equation}
However, since $\tilde{H}_0 \subset \tilde{H}_{\alpha}$,
$\sup_{\tilde{H}_0} (T(\tilde{h})-I(\tilde{h}))$ is at least as
small as $\sup_{\tilde{H}_{\alpha}}(T(\tilde{h})-I(\tilde{h}))$.
Our claim thus follows.
\end{proof}

Fix $e$. Let $\tilde{H}$ be the subset of $\tilde{H}_0$ where
$T(\tilde{h})-I(\tilde{h})$ is maximized. By the compactness of
$\tilde{H}_0$, the continuity of $T$ and the lower semi-continuity
of $I$, $\tilde{H}$ is a nonempty compact set. Theorem \ref{main1}
gives an asymptotic formula for $\psi^{e}_{n,\alpha}$ but says
nothing about the behavior of a typical random graph sampled from
the constrained exponential model (\ref{cpmf}). In the
unconstrained case (\ref{pmf2}) however, we know that the quotient
image $\tilde{h}^{G_n}$ of a sampled graph must lie close to the
corresponding maximizing set $\tilde{H}$ for $\psi$ with
probability vanishing in $n$. We expect that a similar phenomenon
should occur in the constrained model as well, and this is
confirmed by Theorem \ref{main2}.

\begin{theorem}
\label{main2} Take $e\in[0,1]$. Let $\tilde{H}$ be defined as
above. Let $\PR^{e}_{n,\alpha}(G_n)$ (\ref{cpmf}) be the
conditional probability measure on $\mathcal{G}_n$. Then for any
$\eta>0$ and $\alpha$ sufficiently small there exist $C, \gamma>0$
such that for all $n$ large enough,
\begin{equation}
\PR^{e}_{n,\alpha}\left(\delta_{\square}(\tilde{h}^{G_n},
\tilde{H})\geq \eta\right)\leq Ce^{-n^2 \gamma}.
\end{equation}
\end{theorem}

\begin{proof}
We check that the conditional probability measure
$\PR^{e}_{n,\alpha}$ is well defined for all large enough $n$. It
suffices to show that $\psi^{e}_{n,\alpha}$ is finite. But from
(\ref{chain}), $\psi^{e}_{n,\alpha}$ is trapped between
$\sup_{\tilde{h}\in
\tilde{U}_{\alpha}}(T(\tilde{h})-I(\tilde{h}))$ and
$\sup_{\tilde{h}\in
\tilde{H}_{\alpha}}(T(\tilde{h})-I(\tilde{h}))$, which are clearly
both finite.

Recall that $\tilde{H}_{\alpha}$ is the set of reduced graphons
$\tilde{h}$ with $e-\alpha\leq e(\tilde{h})\leq e+\alpha$. Take
any $\eta>0$. Let $\tilde{A}_{\alpha}$ be the subset of
$\tilde{H}_{\alpha}$ consisting of reduced graphons that are at
least $\eta$-distance away from $\tilde{H}$,
\begin{equation}
\tilde{A}_{\alpha}=\{\tilde{h}\in \tilde{H}_{\alpha}:
\delta_{\square}(\tilde{h}, \tilde{H})\geq \eta\}.
\end{equation}
It is easy to see that $\tilde{A}_{\alpha}$ is a closed set.
Without loss of generality we assume that $\tilde{A}_{\alpha}$ is
nonempty for every $\alpha>0$, since otherwise our claim trivially
follows. Under this nonemptiness assumption we can find a sequence
of reduced graphons $\tilde{h}_{\alpha} \in \tilde{A}_{\alpha}$
converging to a reduced graphon $\tilde{h}_0 \in \tilde{A}_0$,
which shows that $\tilde{A}_0$ is nonempty as well. By the
compactness of $\tilde{H}_0$ and $\tilde{H}$, and the upper
semi-continuity of $T-I$, it follows that
\begin{equation}
\max_{\tilde{h}\in
\tilde{H}_0}(T(\tilde{h})-I(\tilde{h}))-\max_{\tilde{h}\in
\tilde{A}_0}(T(\tilde{h})-I(\tilde{h}))>0.
\end{equation}
From the proof of Theorem \ref{main1} we see that
\begin{equation}
\limsup_{\tilde{H}_{\alpha}}(T(\tilde{h})-I(\tilde{h}))=\max_{\tilde{H}_0}(T(\tilde{h})-I(\tilde{h})).
\end{equation}
Similarly, we have
\begin{equation}
\limsup_{\tilde{A}_{\alpha}}(T(\tilde{h})-I(\tilde{h}))=\max_{\tilde{A}_0}(T(\tilde{h})-I(\tilde{h})).
\end{equation}
This implies that for $\alpha$ sufficiently small,
\begin{equation}
2\gamma:=\sup_{\tilde{h}\in
\tilde{H}_{\alpha-\alpha^2}}(T(\tilde{h})-I(\tilde{h}))-\sup_{\tilde{h}\in
\tilde{A}_{\alpha}}(T(\tilde{h})-I(\tilde{h}))>0.
\end{equation}

Choose $\epsilon=\gamma$ and define $\tilde{H}_{\alpha,c}$ and $R$
as in the proof of Theorem \ref{main1}. Let
$\tilde{A}_{\alpha,c}=\tilde{A}_{\alpha}\cap
\tilde{H}_{\alpha,c}$. Then
\begin{equation}
\PR^{e}_{n,\alpha}(\tilde{h}^{G_n} \in \tilde{A}_{\alpha})\leq
\exp(-n^2\psi^{e}_{n,\alpha})|R|\sup_{c\in
R}\exp(n^2(c+\gamma))|\tilde{A}_{\alpha,c}^n|.
\end{equation}
While bounding the last term above, it may be assumed without loss
of generality that $\tilde{A}_{\alpha,c}$ is nonempty for each
$c\in R$. Similarly as in the proof of Theorem \ref{main1}, the
above inequality gives
\begin{equation}\label{supsup}
\limsup_{n\rightarrow \infty}\frac{\log
\PR^{e}_{n,\alpha}(\tilde{h}^{G_n}\in
\tilde{A}_{\alpha})}{n^2}\leq \sup_{c\in
R}\left(c+\gamma-\inf_{\tilde{h}\in
\tilde{A}_{\alpha,c}}I(h)\right)-\sup_{\tilde{h}\in
\tilde{H}_{\alpha-\alpha^2}}\left(T(\tilde{h})-I(\tilde{h})\right).
\end{equation}
Each $\tilde{h}\in \tilde{A}_{\alpha,c}$ satisfies
$T(\tilde{h})\geq c$. Consequently,
\begin{equation}
\sup_{\tilde{h}\in
\tilde{A}_{\alpha,c}}(T(\tilde{h})-I(\tilde{h}))\geq
c-\inf_{\tilde{h}\in \tilde{A}_{\alpha,c}}I(\tilde{h}).
\end{equation}
Substituting this in (\ref{supsup}) gives
\begin{equation}
\limsup_{n\rightarrow \infty}\frac{\log
\PR^{e}_{n,\alpha}(\tilde{h}^{G_n}\in
\tilde{A}_{\alpha})}{n^2}\leq \gamma+\sup_{\tilde{h}\in
\tilde{A}_{\alpha}}(T(\tilde{h})-I(\tilde{h}))-\sup_{\tilde{h}\in
\tilde{H}_{\alpha-\alpha^2}}(T(\tilde{h})-I(\tilde{h}))=-\gamma.
\end{equation}
This completes the proof.
\end{proof}

\section{An application}
\label{3} Theorems \ref{main1} and \ref{main2} in the previous
section illustrate the importance of finding the maximizing
graphons for $T-I$ subject to certain constraints. Similar
optimization problems have also been studied in the context of
upper tails of random graphs by Lubetzky and Zhao \cite{LZ}. The
optimizers aid us in understanding the limiting conditional
probability distribution and the global structure of a random
graph $G_n$ drawn from the constrained exponential model. Indeed,
knowledge of such graphons would help us understand the limiting
probability distribution and the global structure of a random
graph $G_n$ drawn from the unconstrained exponential model as
well, since we can always carry out the unconstrained optimization
in steps: first consider a constrained optimization (referred to
as ``micro analysis''), then take into consideration of all
possible constraints (referred to as ``macro analysis''). However,
as straight-forward as it sounds, due to the myriad of structural
possibilities of graphons, both the unconstrained (\ref{setmax})
and constrained (\ref{csetmax}) optimization problems are not
always explicitly solvable. So far major simplification has only
been achieved in the ``attractive'' case where the parameters
$\beta_2,\ldots,\beta_k$ are all nonnegative \cite{CD} \cite{RY}
\cite{Y} and for $k$-star models \cite{CD}, whereas a complete
analysis of either (\ref{setmax}) or (\ref{csetmax}) in the
``repulsive'' region where the parameters $\beta_2,\ldots,\beta_k$
are all negative has proved to be very difficult. This section
will provide some phase transition results on the constrained
``repulsive'' edge-triangle exponential random graph model and
discuss their possible generalizations. Using the same arguments,
it is also possible to establish the phase transition in the
``attractive'' region of the parameter space. We make these
notions precise in the following.

The unconstrained edge-triangle model is a $2$-parameter
exponential random graph model obtained by taking $H_1$ to be a
single edge and $H_2$ to be a triangle in (\ref{pmf}). More
explicitly, in the edge-triangle model, the probability measure
$\PR_n^{\beta}$ is
\begin{equation}
\label{etpmf} \PR_n^{\beta}(G_n)=\exp\left(n^2(\beta_1 e(G_n)+
\beta_2 t(G_n)-\psi_n^{\beta})\right),
\end{equation}
where $\beta=(\beta_1,\beta_2)$ are $2$ real parameters, $e(G_n)$
and $t(G_n)$ are the edge and triangle densities of $G_n$, and
$\psi_n^\beta$ is the normalization constant. As before, we assume
that the ideal edge density $e$ is fixed. The limiting
construction described at the beginning of Section \ref{2} will
then yield the asymptotic conditional normalization constant
$\psi^e$. From (\ref{csetmax}) we see that $\psi^e$ depends on
both parameters $\beta_1$ and $\beta_2$, however the $\beta_1$
dependence is linear: $\psi^e$ is equal to $\beta_1 e$ plus a
function independent of $\beta_1$. In particular $\beta_1$ plays
no role in the maximization problem, so we can consider it fixed
at value $\beta_1=0$. The only relevant parameters then are $e$
and $\beta_2$.

To highlight this parameter dependence, in the following we will
write $\psi^e$ as $\psi^{e,\beta_2}$ instead. We are particularly
interested in the asymptotics of $\psi^{e,\beta_2}$ when $\beta_2$
is negative, the so-called repulsive region. Naturally, varying
$\beta_2$ allows one to adjust the influence of the triangle
density of the graph on the probability distribution. The more
negative the $\beta_2$, the more unlikely that graphs with a large
number of triangles will be observed. When $\beta_2$ approaches
negative infinity, the most probable graph would likely be
triangle free. At the other extreme, when $\beta_2$ is zero, the
edge-triangle model reduces to the well-studied
Erd\H{o}s-R\'{e}nyi model, where edges between different vertex
pairs are independently included. The structure of triangle free
graphs and disordered Erd\H{o}s-R\'{e}nyi graphs are apparently
quite different, and thus a phase transition is expected as
$\beta_2$ decays from $0$ to $-\infty$. In fact, it is believed
that, quite generally, repulsive models exhibit a transition
qualitatively like the solid/fluid transition, in that a region of
parameter space depicting emergent multipartite structure, which
is in imitation of the structure of solids, is separated by a
phase transition from a region of disordered graphs, which
resemble fluids. The existence of such a transition in
unconstrained $2$-parameter models whose subgraph $H_2$ has
chromatic number at least $3$ has been proved by Aristoff and
Radin \cite{AR} based on a symmetry breaking result from
\cite{CD}. Theorem \ref{csep} below gives a corresponding result
in the constrained edge-triangle model. Its proof though is quite
different from the parallel result in \cite{AR} and relies instead
on some analysis arguments.

\begin{theorem}
\label{csep} Consider the constrained repulsive edge-triangle
exponential random graph model as described above. Let $e$ be
arbitrary but fixed. Let $\beta_2$ vary from $0$ to $-\infty$.
Then $\psi^{e,\beta_2}$ is not analytic at at least one value of
$\beta_2$.
\end{theorem}

\begin{proof}
We first consider the case $e\le 1/2$; the case $e>1/2$ is
similar, see the comments at the end of the proof.

Let $e(\tilde{h})\le1/2$ be the edge density of a reduced graphon
$\tilde{h}$ and $t(\tilde{h})$ be the triangle density, obtained
by taking $H$ to be a triangle in (\ref{tt}). By (\ref{csetmax}),
\begin{eqnarray}
\label{etsetmax} \psi^{e,\beta_2}&=&\sup_{\tilde{h}\in
\tilde{\mathcal{W}}: e(\tilde{h})=e}\left(\beta_2
t(\tilde{h})-I(\tilde{h})\right)\\\notag
&=&\sup_t\sup_{\tilde{h}\in \tilde{\mathcal{W}}: e(\tilde{h})=e,
t(\tilde{h})=t}\left(\beta_2 t-I(\tilde{h})\right)\\\notag
&=&\sup_t \left(\beta_2 t+s(e, t)\right),
\end{eqnarray}
where for notational convenience, we denote by $s(e, t)$ the
maximum value of $-I(\tilde{h})$ over all reduced graphons with
$e(\tilde{h})=e$ and $t(\tilde{h})=t$. We examine (\ref{etsetmax})
at the two extreme values of $\beta_2$ first. Since $I$ is convex,
when $\beta_2=0$,
\begin{equation}
\psi^{e,0}=\sup_{\tilde{h}\in \tilde{\mathcal{W}}:
e(\tilde{h})=e}\left(-I(\tilde{h})\right)\leq -I(e)
\end{equation}
by Jensen's inequality, and the equality is attained only when $h
\equiv e$, the associated graphon for an Erd\H{o}s-R\'{e}nyi graph
with edge formation probability $e$. This also ensures that when
we take $\beta_2 \leq 0$, any maximizing graphon $h$ for
(\ref{etsetmax}) will satisfy $t(\tilde{h})\leq e^3$. For the
other extreme, take an arbitrary sequence $\beta_2^{(i)}
\rightarrow -\infty$, and let $\tilde{h}_i$ be a maximizing
reduced graphon for each $\psi^{e,\beta_2^{(i)}}$. Let $\tilde{h}$
be a limit point of $\tilde{h}_i$ in $\widetilde{\mathcal{W}}$
(its existence is guaranteed by the compactness of
$\tilde{\mathcal{W}}$). We say that a graphon $h:
[0,1]^2\rightarrow [0,1]$ is symmetric bipodal if it is of the
form
\begin{equation}
h(x, y)=\left\{%
\begin{array}{ll}
    p & \hbox{if $x<1/2<y$ or $x>1/2>y$;} \\
    q & \hbox{if $x, y<1/2$ or $x, y>1/2$,} \\
\end{array}%
\right.
\end{equation}
where $p$ and $q$ are constants taking values between $0$ and $1$.
Suppose $t(\tilde{h})>0$. Then by the continuity of $t$ and the
boundedness of $I$, $\lim_{i \rightarrow
\infty}\psi^{e,\beta_2^{(i)}}=-\infty$. But this is impossible
since $\psi^{e, \beta_2^{(i)}}$ is uniformly bounded below, as can
be seen by considering the symmetric bipodal graphon $h$ with
$p=2e$ and $q=0$ as a test function, which corresponds to a
complete bipartite graph with $1-2e$ fraction of edges randomly
deleted. Thus $t(\tilde{h})=0$. The rest of the proof will utilize
the following useful features of $s(e, t)$ derived in Radin and
Sadun \cite{RS1} \cite{RS}. From the convexity of $I$, Theorem 4.1
in \cite{RS1} finds that for $e\le 1/2$, $s(e, 0)=-I(2e)/2$ and
this maximum is achieved only at the reduced symmetric bipodal
graphon $\tilde{h}$ depicted above. Further utilizing properties
of the Hermitian trace class operator, Theorem 1.1 in \cite{RS}
states that for any $e\in[0,1]$ and for $t\leq e^3$,
\begin{equation}\label{2/3}
s(e, e^3)-s(e, t)\geq c(e^3-t)^{2/3}
\end{equation}
for some $c=c(e)>0$. Thus we have
\begin{equation}
\label{lim1} \lim_{\beta_2 \rightarrow
-\infty}\psi^{e,\beta_2}=-I(2e)/2;
\end{equation}
while (\ref{2/3}) implies that for $\beta_2>-c(e)$ and $t<e^3$,
\begin{equation}
-\beta_2(e^3-t)<s(e, e^3)-s(e, t).
\end{equation}
In other words, the constant graphon $h \equiv e$ still yields the
maximum value for (\ref{etsetmax}) for these small values of
$\beta_2$. Thus regarded as a function of $\beta_2$,
$\psi^{e,\beta_2}$ is constant on the interval $(-c(e),0)$ and
$\psi^{e,\beta_2}=-I(e)$. This shows that $\psi^{e,\beta_2}$ must
lose its analyticity at at least one $\beta_2$ as $\beta_2$ varies
from $0$ to $-\infty$, since otherwise we would have
\begin{equation}
\lim_{\beta_2 \rightarrow -\infty}\psi^{e,\beta_2}=-I(e),
\end{equation}
in contradiction with (\ref{lim1}).

For $e>1/2$, the lower boundary of attainable $t(\tilde h)$ is
nonzero; see Figure \ref{razfig}. However the graphons attaining
the minimum $t$ values for each $e$ are known, see \cite{RS1}, and
their rate functions are strictly less than $-I(e),$ so the proof
above goes through without change.
\end{proof}

\begin{figure}[htbp]
\center\includegraphics[width=3in]{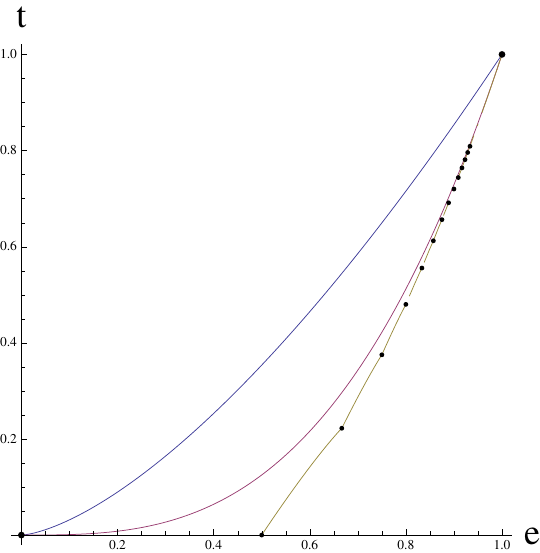}
\caption{\label{razfig} Region of attainable edge ($e$) and
triangle ($t$) densities for graphons. The upper boundary is the
curve $t=e^{3/2}$ and the lower boundary is a piecewise algebraic
curve with infinitely many concave pieces; see \cite{Raz}. The red
curve is the Erd\H{o}s-Renyi curve $t=e^3$.}
\end{figure}

The proof of Theorem \ref{csep} does not rely heavily on the
definition of the edge-triangle model, except for the
non-differentiability of $s(e, t)$ at $t=e^3$ and the structure of
the maximizing graphons at the two extreme values of $\beta_2$.
The following extension of this theorem may not come as a
surprise.

\begin{theorem}
Take $H_1$ a single edge and $H_2$ a different, arbitrary simple
graph with chromatic number $\chi(H_2)$ at least $3$. Consider the
constrained repulsive $2$-parameter exponential random graph model
where the probability measure $\PR_n^{e,\beta_2}$ is given by
\begin{equation}
\label{2pmf} \PR_n^{e,\beta_2}(G_n)=\exp\left(n^2(\beta_2 t(H_2,
G_n)-\psi_n^{e,\beta_2})\right).
\end{equation}
Let the edge density $e$ be fixed. Let the second parameter
$\beta_2$ vary from $0$ to $-\infty$. Then $\psi^{e,\beta_2}$
loses its analyticity at at least one value of $\beta_2$.
\end{theorem}

\begin{proof}
The proof of Theorem \ref{csep} carries over almost word-for-word
when we incorporate the disordered Erd\H{o}s-R\'{e}nyi structure
of the maximizing graphon at $\beta_2=0$, the
non-differentiability of $s(e, t)$ for a general $H_2$ \cite{RS},
and the emergent multipartite structure of the maximizing graphon
as $\beta_2 \rightarrow -\infty$ \cite{CD} \cite{YRF}.
\end{proof}

Now that we know about the occurrence of a phase transition in the
constrained repulsive exponential model, we probe deeper into this
phenomenon and ask: how smooth is this transition? Theorem
\ref{etmain} shows what happens when the ideal edge density of the
edge-triangle model is fixed at $1/2$ while the influence of the
triangle densities is tuned through the parameter $\beta_2$.

\begin{theorem}
\label{etmain} Consider the constrained repulsive edge-triangle
exponential random graph model as described at the beginning of
Section \ref{3}. Fix $e=1/2$. Let $\beta_2$ vary from $0$ to
$-\infty$. Then $\psi^{\frac{1}{2},\beta_2}$ is analytic everywhere except
at a certain point $\beta_2^c$, where the derivative
$\frac{\partial}{\partial \beta_2}\psi^{\frac{1}{2},\beta_2}$ displays jump
discontinuity.
\end{theorem}

\begin{figure}[htbp]
\center\includegraphics[width=3in]{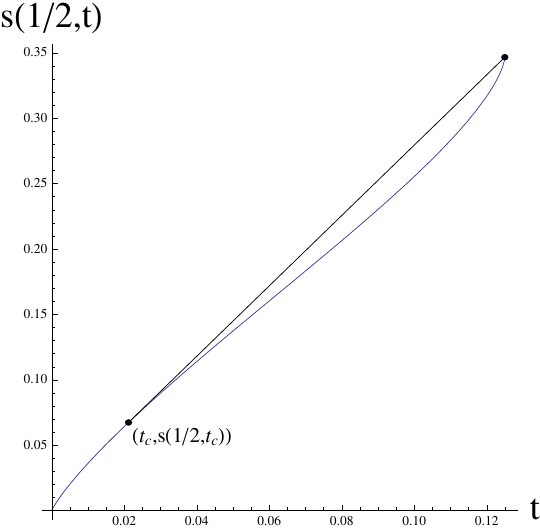}
\caption{\label{nonconvex}The graph of $s(1/2,t)$ below the ER
line for the edge-triangle model (blue). The convex hull of the
region below the graph is delimited by the black line segment and
the portion of the graph to its left; this segment is the support
line at the right endpoint $(t=1/8,s=\frac{\log 2}{2})$ of maximal
slope $-\beta_2^c$. The other point at which the line segment
meets the curve is the point $(t_c,s(1/2,t_c))$.}
\end{figure}

\begin{proof}
Setting $e=1/2$ in (\ref{etsetmax}) gives
\begin{equation}
\label{12setmax} \psi^{\frac{1}{2},\beta_2}=\sup_t (\beta_2
t+s(\frac{1}{2},t)).
\end{equation}
Since $\beta_2 \leq 0$, by the convexity of $I$, any maximizing
graphon $h$ for (\ref{12setmax}) must satisfy $t(\tilde{h}) \leq
1/8$, i.e., it must lie below the Erd\H{o}s-R\'{e}nyi curve
$t=e^3$. Radin and Sadun \cite{RS} showed that on the line segment
$e=\frac{1}{2}$ and $t\leq e^3$, the symmetric bipodal graphon
\begin{equation}
\label{bipodal}
h(x, y)=\left\{%
\begin{array}{ll}
    \frac{1}{2}+\epsilon, & \hbox{if $x<\frac{1}{2}<y$ or $x>\frac{1}{2}>y$;} \\
    \frac{1}{2}-\epsilon, & \hbox{if $x, y<\frac{1}{2}$ or $x, y>\frac{1}{2}$,} \\
\end{array}%
\right.
\end{equation}
where $0\leq \epsilon=(\frac{1}{8}-t)^{\frac{1}{3}}\leq
\frac{1}{2}$, maximizes $s(\frac{1}{2},t)$, and that every
maximizing graphon is of the form $h_{\sigma}$ for some measure
preserving bijection $\sigma$. Equivalently, the maximum value for
(\ref{12setmax}) is achieved only at the reduced bipodal graphon
$\tilde{h}$. See Figure \ref{nonconvex} for the graph of
$s(1/2,t)$.

\begin{figure}[htbp]
\center\includegraphics[width=3in]{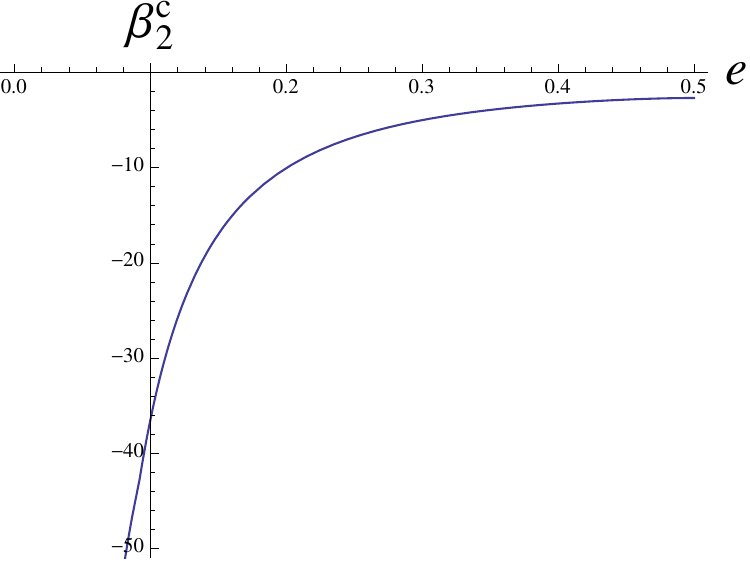}
\caption{\label{belowER1} The (conjectural) graph of $\beta_2^c$
as a function of $e$ for the edge-triangle model, in the range
$e\le 1/2$. The computation is based on the conjecture that the
maximizing graphons in this region are symmetric and bipodal, see
\cite{RRS}.}
\end{figure}

Geometrically, the maximization problem in (\ref{12setmax})
involves finding the lowest half-plane with bounding line of slope
$-\beta_2$ lying above the graph of $s(1/2,t)$. For
$\beta_2>\beta_2^c$ the boundary of this half-plane passes only
through the graph of $s(1/2,t)$ at the right endpoint
$(\frac{1}{8},\frac{\log 2}{2})$. The critical value $\beta_2^c$
is defined (as in Figure \ref{nonconvex}) as the first slope at
which this half-plane intersects the curve at a different point.
We let $(t_c,s(1/2,t_c))$ be this second point. At more negative
values of $\beta_2$, the half-plane will hit the curve at points
with $t$ values below $t_c$.

In particular this shows the non-analyticity of
$\psi^{\frac{1}{2}, \beta_2}$ as a function of $\beta_2$ at
$\beta_2=\beta_2^c$. The analyticity of $\psi^{\frac{1}{2},
\beta_2}$ elsewhere follows from concavity (and analyticity) of
$s(1/2,t)$ below $t_c$. By Theorem \ref{main2}, at
$\beta_2=\beta_2^c$, the maximizing reduced graphon $\tilde{h}$
for (\ref{12setmax}) transitions from being Erd\H{o}s-R\'{e}nyi
with edge formation probability $\frac{1}{2}$ to symmetric bipodal
with $\eps_c=(\frac18-t_c)^{1/3}$. The jump discontinuity in the
derivative follows when we realize that $\frac{\partial}{\partial
\beta_2}\psi^{\frac{1}{2},\beta_2}=t(\tilde{h})$.
\end{proof}

Numerical computations yield that $\beta_2^c$ is approximately
$-2.7$ and $\eps_c$ is approximately $0.47$. By Theorem
\ref{main2}, this shows that as $\beta_2$ decreases from $0$ to
$-\infty$, a typical graph $G_n$ drawn from the constrained
repulsive edge-triangle model jumps from being Erd\H{o}s-R\'{e}nyi
to almost complete bipartite, skipping a large portion of the
$e=\frac{1}{2}$ line. This ``jump behavior'' (also called
first-order phase transition) is intrinsically tied to the
\emph{convexity} of $s(e, t)$ just below the Erd\H{o}s-R\'{e}nyi
curve $t=e^3$, thus we expect similar phase transition phenomena
for general $e \neq \frac{1}{2}$ as well; see Figures
\ref{belowER1}, \ref{belowER2}. However, unlike in the $e=\frac{1}{2}$ case, where the symmetry of $I$ (\ref{I1}) about $u=\frac{1}{2}$ contributes to a precise knowledge of the structure of the maximizing graphon, in general cases, there is only empirical evidence concerning the structure of the maximizing graphons. See also \cite{CD} for related results
in the unconstrained repulsive edge-triangle model.

\begin{figure}[htbp]
\center\includegraphics[width=3in]{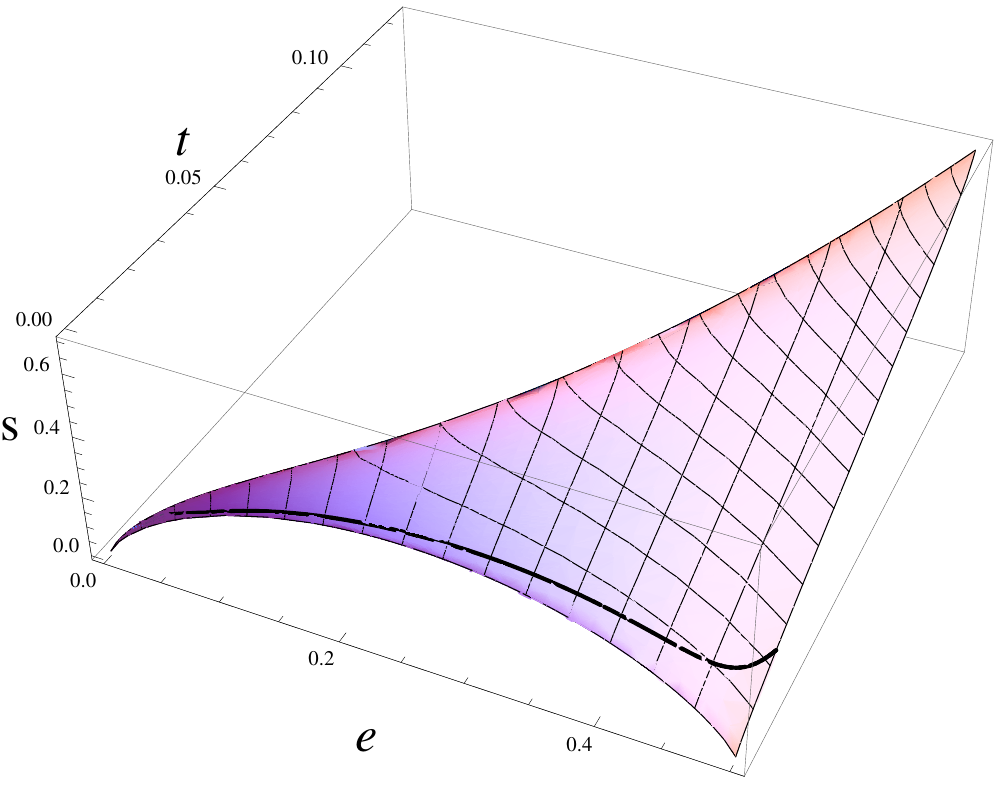}
\caption{\label{belowER2} The (conjectural) graph of entropy $-I$
as a function of edge and triangle densities $e,t$ in the region
$e\le\frac12, t\le e^3$. The critical curve (black) defines $t^c$
as a function of $e$. The computation is based on the conjecture
that the maximizing graphons in this region are symmetric and
bipodal, see \cite{RRS}.}
\end{figure}

\section{Euler-Lagrange equations}
\label{4} We return to the constrained $2$-parameter family of
exponential random graphs (\ref{2pmf}). For notational convenience
and with some abuse of notation, denote by $T(h)=\sum_{i=1}^2
\beta_i t(H_i, h)$. As seen in Section \ref{3}, the ``micro
analysis'' helps with the ``macro analysis''. Explicitly, if we
can find the maximizing graphon for $-I$ subject to two
constraints $t(H_1, \cdot)=t_1$ and $t(H_2, \cdot)=t_2$, where
$t_1$ and $t_2$ are arbitrary but fixed homomorphism densities,
then we can find the maximizing graphon for $T-I$ subject to fewer
or even no constraints. This in turn will aid us in understanding
the limiting conditional probability distribution and the
structure of a typical graph $G_n$ sampled from either the
constrained or the unconstrained exponential model. In the
unconstrained case, Chatterjee and Diaconis derived the
Euler-Lagrange equation for the maximizing graphon $h$ for
$T(h)-I(h)$ when the tuning parameters are arbitrary but fixed
(Theorem 6.1 in \cite{CD}). When applied to the $2$-parameter
model, they showed that $h$ must be bounded away from $0$ and $1$ and for almost all $(x,y)\in [0,1]^2$,
\begin{equation}
\label{lagrange} h(x,y)=\frac{e^{2\sum_{i=1}^2\beta_i
\Delta_{H_i}h(x,y)}}{1+e^{2\sum_{i=1}^2\beta_i
\Delta_{H_i}h(x,y)}},
\end{equation}
where for a finite simple graph $H$ with vertex set $V(H)$ and
edge set $E(H)$,
\begin{equation}
\Delta_H h(x,y)=\sum_{(r,s)\in E(H)}\Delta_{H, r, s}h(x, y),
\end{equation}
and for each $(r,s) \in E(H)$ and each pair of points $x_r,x_s \in
[0, 1]$,
\begin{multline}
\Delta_{H, r, s}h(x_r,
x_s)=\\\int_{[0,1]^|V(H)\texttt{\char92}\{r,s\}|}\prod_{(r',s')\in
E(H):(r',s')\neq (r,s)}h(x_{r'},x_{s'})\prod_{v\in V(H):v\neq
r,s}dx_v.
\end{multline}
For example, in the edge-triangle model where $H_1$ is an edge and
$H_2$ is a triangle, $\Delta_{H_1}h(x,y)\equiv 1$ and
$\Delta_{H_2}h(x,y)=3\int_0^1 h(x,z)h(y,z)dz$. In the constrained
case, we could likewise derive the Euler-Lagrange equation by
resorting to the method of Lagrange multipliers, which will turn
the constrained maximization into an unconstrained one, but we
provide an alternative bare-hands approach here. The following
theorem may also be formulated in terms of reduced graphons.

\begin{theorem}
Consider the constrained $2$-parameter exponential random graph
model (\ref{2pmf}). Let $t_1$ and $t_2$ be arbitrary but fixed
homomorphism densities. Suppose the graphon $h$ maximizes $-I(h)$
subject to $t(H_1,h)=t_1$ and $t(H_2,h)=t_2$. If $h$ is bounded
away from $0$ and $1$, then there must exist constants $\beta_1$
and $\beta_2$ such that $h$ satisfies (\ref{lagrange}) for almost
all $(x,y)\in [0,1]^2$.
\end{theorem}

\begin{proof}
Graphons are bounded integrable functions on $[0,1]^2$ so they are continuous outside a set of arbitrarily small measure. Let $(x_i,y_i)$ for $i=1,2,3$ be
three points of $[0,1]^2$. Inside a very small ball near $(x_i,y_i)$, write $h=h_i+\bar{h}$ where $h_i$ is the average of $h$ in that ball. We
infinitesimally perturb the values of $h$ around $(x_i,y_i)$, sending
$h_i\to h_i+dh_i$. Since $\bar{h}$ averages to $0$ and is pointwise small, in computing $t_1,t_2,$ and $-I$,
terms involving $\bar{h}$ only contribute to second order and may be ignored in the computation below. Then $(t_1, t_2, -I) \rightarrow (t_1,t_2,
-I)+(dt_1,dt_2,-dI)$ where
\begin{equation}
\left(%
\begin{array}{c}
  dt_1 \\
  dt_2 \\
  -dI   \\
\end{array}%
\right)=\left(%
\begin{array}{cccc}
  \Delta_{H_1}h_1 & \Delta_{H_1}h_2 & \Delta_{H_1}h_3 \\
  \Delta_{H_2}h_1 & \Delta_{H_2}h_2 & \Delta_{H_2}h_3 \\
  \frac{1}{2}\log(\frac{1}{h_1}-1) & \frac{1}{2}\log(\frac{1}{h_2}-1) & \frac{1}{2}\log(\frac{1}{h_3}-1) \\
\end{array}%
\right)\left(%
\begin{array}{c}
  dh_1 \\
  dh_2 \\
  dh_3 \\
\end{array}%
\right).
\end{equation}
If the determinant of the above matrix is nonzero, then there is a
nontrivial deformation $(dh_1,dh_2,dh_3)$ which increases $-I$
while leaving $t_1$ and $t_2$ fixed. So the maximizing graphon $h$
must satisfy the condition that the determinant is zero. Recall
that $H_1$ is a single edge and $\Delta_{H_1}h_i \equiv 1$. Without loss of generality we assume that $h_1 \neq h_2$, since otherwise $h$
is a constant graphon and our claim trivially follows. Thus
the first and third rows of the matrix are linearly independent
and there must exist constants $\beta_1$ and $\beta_2$ such that
\begin{equation}
\label{univ}
\Delta_{H_2}h_i=\beta_1+\frac{\beta_2}{2}\log(\frac{1}{h_i}-1).
\end{equation}
Moreover, since $\beta_1$ and $\beta_2$ are determined by $h_1$
and $h_2$, we must have (\ref{univ}) for all points $(x_3,y_3)\in
[0,1]^2$. We recognize this requirement is equivalent to
(\ref{lagrange}).
\end{proof}

Suppose we are looking for a graphon $h$ that maximizes $-I(h)$
subject to $t(H_1,h)=t_1$ only. Then following the same
``perturbation'' idea, we should examine
\begin{equation}
\left(%
\begin{array}{c}
  dt_1 \\
  -dI   \\
\end{array}%
\right)=\left(%
\begin{array}{cccc}
  \Delta_{H_1}h_1 & \Delta_{H_1}h_2 \\
  \frac{1}{2}\log(\frac{1}{h_1}-1) & \frac{1}{2}\log(\frac{1}{h_2}-1) \\
\end{array}%
\right)\left(%
\begin{array}{c}
  dh_1 \\
  dh_2 \\
\end{array}%
\right).
\end{equation}
Since the determinant is zero, $h$ must be a constant. This is the
same conclusion obtained by applying Jensen's inequality to the
convex function $I$. On the other hand, we may also consider
maximizing $-I(h)$ subject to $k$ (instead of $2$) constraints
$t(H_i,h)=t_i$ for $i=1,\ldots,k$, in which case we would perturb
the values of the graphon at $k+1$ points and form a $(k+1)\times
(k+1)$ matrix.

\medskip

\acks

Richard Kenyon's research was partially supported by NSF grant DMS-1208191 and a Simons Investigator award.
Mei Yin's research was partially supported by NSF grant
DMS-1308333. They thank Charles Radin, Kui Ren, and Lorenzo Sadun for helpful
conversations.

\end{document}